\documentclass[a4paper]{amsart}
\usepackage{graphicx}
\usepackage{amssymb}
\usepackage{amsmath}
\usepackage{amsthm}
\usepackage{amscd}
\usepackage[all,2cell]{xy}
\usepackage{etex}
\usepackage{tikz}
\usetikzlibrary {arrows.meta}
\usepackage[pagebackref,colorlinks]{hyperref}
\usepackage{float}

\UseAllTwocells \SilentMatrices
\newtheorem{thm}{Theorem}[section]
\newtheorem{alphatheorem}{Theorem}

\newtheorem{cor}[thm]{Corollary}

\newtheorem{lem}[thm]{Lemma}
\newtheorem{exm}[thm]{Example}

\newtheorem{prop}[thm]{Proposition}

\theoremstyle{definition}
\newtheorem{defn}[thm]{Definition}
\theoremstyle{remark}
\newtheorem{rem}[thm]{\bf Remark}
\numberwithin{equation}{section}

\begin{document}
\title[Frobenius quotients, inflation categories and wPL]{Frobenius quotients, inflation categories and weighted projective lines}
\author[Xiao-Wu Chen, Qiang Dong, Shiquan Ruan] {Xiao-Wu Chen, Qiang Dong, Shiquan Ruan}

\makeatletter
\@namedef{subjclassname@2020}{\textup{2020} Mathematics Subject Classification}
\makeatother

\date{\today}

\thanks{E-mail: xwchen$\symbol{64}$mail.ustc.edu.cn, dongqiang@fjnu.edu.cn, 
sqruan@xmu.edu.cn}

\begin{abstract}
We propose the notion of Frobenius quotients between Frobenius exact categories. It turns out that any Frobenius quotient induces  Frobenius quotients between the corresponding inflation categories. We obtain an explicit Frobenius quotient from the category of vector bundles on weighted projective lines with three weights to a certain category consisting of  monomorphism grids. 
\end{abstract}

\maketitle

\dedicatory{}%
\commby{}%

\section{Introduction}

Weighted projective lines \cite{GL} arise naturally in the geometric study of canonical algebras \cite{Rin}. They are closely related to hereditary abelian categories, algebras of automorphic forms and preprojective algebras; see \cite{Len}.  The work \cite{KLM2} establishes a surprising connection between weighted projective lines and the Birkhoff-type problem \cite{Birk, Arn}. 

To be more precise, we denote by ${\rm vect}\mbox{-}\mathbb{X}(2,3,r)$ the category of vector bundles on the weighted projective line of weight type $(2,3,r)$. Denote by $\mathcal{S}^\mathbb{Z}(r)$ the category of graded 
 invariant subspaces of nilpotent operators with nilpotency index at most $r$, that is, the monomorphism category \cite{RS08, Sim} of $\mathbb{Z}$-graded modules over the truncated polynomial algebra $\mathbb{K}[t]/(t^r)$. Both ${\rm vect}\mbox{-}\mathbb{X}(2,3,r)$ and $\mathcal{S}^\mathbb{Z}(r)$ are naturally Frobenius exact categories \cite{Qui73, Hap}. The main result in \cite{KLM2} states that there is a certain quotient functor  
 $$F^{\rm KLM}\colon {\rm vect}\mbox{-}\mathbb{X}(2,3,r)\longrightarrow \mathcal{S}^\mathbb{Z}(r),$$
 which induces a stable equivalence. 

We mention that the quotient functor $F^{\rm KLM}$ is rather mysterious, whose construction is somewhat  indirect; compare \cite{Chen12}. Moreover, as mentioned in \cite{KLM1, Sim15}, such a quotient functor might exist in a slightly more general setting, that is, there is  a quotient functor 
 $${\rm vect}\mbox{-}\mathbb{X}(p,q,r)\longrightarrow \mathcal{S}_{p-1, q-1}^\mathbb{Z}(r).$$
 Here, $\mathbb{X}(p,q,r)$ denotes the weighted projective line of weight type $(p,q, r)$,  and $\mathcal{S}_{p-1, q-1}^\mathbb{Z}(r)$ denotes the category of $(p-1)\times (q-1)$-grids consisting of certain monomorphisms between graded modules over $\mathbb{K}[t]/(t^r)$; see Example~\ref{exm:grid}. We emphasize that there is no detailed proof about the existence in the literature. 

 The goal of this work is to confirm the existence of such a quotient functor; see Theorem~\ref{thm:WPL}. 

\vskip 5pt
\begin{alphatheorem} 
\label{thm:A}
 There is an explicit Frobenius quotient functor
 $${\rm vect}\mbox{-}\mathbb{X}(p,q,r)\longrightarrow \mathcal{S}_{p-1, q-1}^\mathbb{Z}(r),$$
 whose essential kernel equals 
 $${\rm add}\; \{\mathcal{O}(i\vec{x}+n\vec{z}), \mathcal{O}(j\vec{y}+n\vec{z})\; |\; 0\leq i\leq p-1, 0\leq j\leq q-1, n\in \mathbb{Z}\}.$$
\end{alphatheorem}

 Here, $\mathcal{O}$ denotes the structure sheaf of $\mathbb{X}(p, q,r)$, and the vectors in the parentheses denote the corresponding elements in the Picard group \cite{GL}.  

 For Theorem~\ref{thm:A}, we propose the notion of \emph{Frobenius quotient} in Definition~\ref{defn:Fq}, which is a certain exact functor between two Frobenius exact categories. The treatment here is inspired by \cite{DI}. The construction of the functor in Theorem~\ref{thm:A} is explicit, since it is essentially given by the so-called \emph{double-cokernel} in Section~\ref{sec:WPL}. 

There are two key steps in the proof of Theorem~\ref{thm:A}. The first step is to relate graded maximal Cohen-Macaulay modules over the homogeneous coordinate algebra of $\mathbb{X}(p, q,r)$ with certain \emph{multifold projective-module factorizations} \cite{BHU, Tri, SZ}; see Section~\ref{sec:PMF}. We mention that such factorizations are natural generalizations of matrix factorizations in \cite{Ei80}; see also \cite{CCKM, MU21, BE}.  Moreover, such factorizations are known as the $p$-cycle construction in \cite{Len}; see also \cite{LO, Ruan}. The second one is to prove the following general result on Frobenius quotients; see Theorem~\ref{thm:inf}. 

\vskip 5pt

\begin{alphatheorem} 
\label{thm:B}  Let $F\colon \mathcal{A} \rightarrow \mathcal{B}$ be a Frobenius quotient between two Frobenius exact categories. Then for each $n\geq 1$, the induced functor ${\rm Inf}_n(F)\colon {\rm Inf}_n(\mathcal{A}) \rightarrow {\rm Inf}_n(\mathcal{B})$ is also a Frobenius quotient.
\end{alphatheorem} 

Here, for any Frobenius category $\mathcal{A}$, we denote by ${\rm Inf}_n(\mathcal{A})$ the category of $n$-inflations in $\mathcal{A}$; it is naturally a Frobenius exact category. If $\mathcal{A}$ is taken to be the category of finite-dimensional  graded modules over $\mathbb{K}[t]/(t^r)$, then ${\rm Inf}_2(\mathcal{A})=\mathcal{S}^\mathbb{Z}(r)$.

Since the two steps above work well in general, they enable us to prove a very general result, which claims a Frobenius quotient involving certain graded Gorenstein-projective modules; see Theorem~\ref{thm:main}. Here, we mention that Gorenstein-projective modules \cite{EJ} are natural generalizations of maximal Cohen-Macaulay modules over Gorenstein rings; see \cite{ABr69}.

The paper is organized as follows. In Section~2, we  introduce Frobenius quotients and prove a recognition theorem; see Theorem~\ref{thm:Frob}. In Section~\ref{sec:inf}, we  define the categories ${\rm Inf}_{m, n}(\mathcal{A})$ of $(m,n)$-inflations in any Frobenius category $\mathcal{A}$. We prove Theorem~\ref{thm:B}. In Section~\ref{sec:PMF}, we relate graded Gorenstein-projective modules with certain multifold projective-module factorizations. We recall the cokernel functor in \cite{SZ}.

In Section~\ref{sec:main}, we apply results in Sections~\ref{sec:inf} and \ref{sec:PMF} to prove the main result,  which claims a Frobenius quotient involving graded Gorenstein-projective modules over two different graded rings; see Theorem~\ref{thm:main}. In the final section, we apply these results to weighted projective lines and prove  Theorem~\ref{thm:A}.

\section{Pretriangle-equivalences and Frobenius quotients}

In this section, we recall basic facts on exact categories and Frobenius categories. Inspired by \cite{Chen12, DI}, we introduce the notion of Frobenius quotients.

Let $\mathcal{A}$ be an additive category. For any additive subcategory $\mathcal{X}$, we denote by $[\mathcal{X}]$ the ideal formed by all the morphisms factoring through $\mathcal{X}$, and by $\mathcal{A}/{[\mathcal{X}]}$ the factor category.

Let $F\colon \mathcal{A}\rightarrow \mathcal{B}$ be an additive functor between two additive categories. The \emph{essential kernel} ${\rm Ker}(F)$ of $F$ is defined to be the full subcategory of $\mathcal{A}$ formed by those objects that are annihilated by $F$.  Following \cite[Appendix]{RZ}, the functor $F$ is called \emph{objective} if each morphism $f$ in $\mathcal{A}$ satisfying $F(f)=0$ necessarily factors through ${\rm Ker}(F)$.

The following fact is well known; consult \cite[Appendix]{RZ}.

\begin{lem}\label{lem:objective}
    Let $F\colon \mathcal{A}\rightarrow \mathcal{B}$ be an objective functor. Assume further that $F$ is full and dense. Then it induces an equivalence $\mathcal{A}/{[{\rm Ker}(F)]}\simeq \mathcal{B}$. \hfill $\square$
\end{lem}

Let $\mathcal{A}$ be an additive category. For any two morphisms $i\colon X\rightarrow Y$ and $p\colon Y\rightarrow Z$, we form  a composable pair $(i, p)$. Such a pair is called a \emph{kernel-cokernel pair} if $i$ is a kernel of $p$ and $p$ is a cokernel of $i$. An \emph{exact structure} $\mathcal{E}$ on $\mathcal{A}$ is a chosen class of kernel-cokernel pairs, which satisfies certain natural axioms; see \cite[Appendix~A]{Kel90} and \cite{Bu10}.

An \emph{exact category} $(\mathcal{A}, \mathcal{E})$ \cite{Qui73} is an additive category $\mathcal{A}$ with a specified exact structure $\mathcal{E}$. The pairs $(i, p)$ in $\mathcal{E}$ are called conflations, where $i$ is called an inflation and $p$ is called a deflation.  When the exact structure $\mathcal{E}$ is understood, we might suppress it and simply say that $\mathcal{A}$ is an exact category.

\begin{lem}\label{lem:strictly-i}
Let $\mathcal{A}$ be an exact category. Consider the following commutative square consisting of four inflations.
\begin{align}\label{equ:si}
   \xymatrix{
   X\ar[d]_-{j}\ar[r]^i & Y\ar[d]^-{j'} \\
   X'\ar[r]^{i'} & Y'
   }
\end{align}
Denote by $Y\times_{X}X'$ the pushout of $i$ and $j$. Then the following statements are equivalent.
\begin{enumerate}
    \item The induced morphism $\bar{j'}\colon {\rm Coker}(i)\rightarrow {\rm Coker}(i')$ by $j'$ is an inflation.
    \item The induced morphism $Y\times_{X}X'\rightarrow Y'$ is an inflation.
    \item The induced morphism $\bar{i'}\colon {\rm Coker}(j)\rightarrow {\rm Coker}(j')$ by $i'$ is an inflation.
\end{enumerate}
\end{lem}

The diagram satisfying these equivalent conditions is said to be \emph{strictly inflated}.

\begin{proof}
The equivalence between (1) and (2) follows immediately from \cite[Proposition~2.10 and Corollary~3.6]{Bu10}.  By symmetry, we have the equivalence between (2) and (3).
\end{proof}

\begin{rem}\label{rem:sinf}
    In the strictly inflated diagram (\ref{equ:si}), the object $X$ is identified with the pullback of $j'$ and $i'$. This follows from \cite[Proposition~2.12]{Bu10} and Lemma~\ref{lem:strictly-i}(2). 
\end{rem}

Let $F\colon \mathcal{A}\rightarrow \mathcal{B}$ be an additive functor between two exact categories. It is called \emph{exact}, provided that it sends conflations in $\mathcal{A}$ to conflations in $\mathcal{B}$. By an \emph{exact equivalence}, we mean an exact functor $F$, which is an equivalence of categories and whose quasi-inverse $F^{-1}\colon \mathcal{B}\rightarrow \mathcal{A}$ is also exact.

Recall that an exact category $\mathcal{A}$ is called \emph{Frobenius} if it has enough projectives and enough injectives, and the class of projectives coincides with that of injectives. Denote by $\mathcal{P}$ the full subcategory formed by all projective-injectives in $\mathcal{A}$. The corresponding factor category $\mathcal{A}/{[\mathcal{P}]}$ is usually denoted by $\underline{\mathcal{A}}$,  and called   the \emph{stable category} of $\mathcal{A}$. By \cite[I.2.6~Theorem]{Hap}, it has a canonical triangulated structure.

For convenience, we introduce the following terminology; compare \cite[I.2.8]{Hap}.

\begin{defn}
 Let $\mathcal{A}$ and $\mathcal{B}$ be two Frobenius categories. By a \emph{pretriangle-equivalence}, we mean an exact functor $F\colon \mathcal{A}\rightarrow \mathcal{B}$, which sends projectives to projectives and the induced triangle functor $F\colon \underline{\mathcal{A}}\rightarrow \underline{\mathcal{B}}$ is an equivalence.
\end{defn}

In other words, a pretriangle-equivalence induces a triangle equivalence between the stable categories.

The following facts are standard.

\begin{lem}\label{lem:pte}
Let $F\colon \mathcal{A}\rightarrow \mathcal{B}$ be a pretriangle-equivalence between two Frobenius categories. Assume that $A$ and $A'$ are objects in $\mathcal{A}$. Then the following statements hold.
\begin{enumerate}
\item If $F(A)$ is projective in $\mathcal{B}$, then $A$ is also projective in $\mathcal{A}$.
\item The functor $F$ induces an isomorphism ${\rm Ext}^1_\mathcal{A}(A, A')\simeq {\rm Ext}^1_\mathcal{B}(F(A), F(A'))$.
\end{enumerate}
\end{lem}

\begin{proof}
    (1) follows from the general fact that an object in any Frobenius category is projective if and only if it is isomorphic to the zero object in the stable category.  For (2), we recall the isomorphisms
    $${\rm Ext}^1_\mathcal{A}(A, A')\simeq \underline{\rm Hom}_\mathcal{A}(A, \Sigma(A')) \mbox{ and } {\rm Ext}^1_\mathcal{B}(F(A), F(A'))  \simeq \underline{\rm Hom}_\mathcal{B}(F(A), \Sigma F(A')).$$
    Here, $\Sigma$ denotes the suspension functor. Then the required isomorphism follows from the induced triangle equivalence.
\end{proof}

\begin{prop}\label{prop:pte}
    Let $F\colon \mathcal{A}\rightarrow \mathcal{B}$ be an additive functor between two Frobenius categories, which is an equivalence. Then $F$ is an exact equivalence if and only if it is a pretriangle-equivalence.
\end{prop}

\begin{proof}
    It  suffices to prove the ``if" part.  For this, we have to prove that its quasi-inverse $F^{-1}$ is also exact. This is equivalent to the following statement: for any composable pair $(i, p)$ of morphisms in $\mathcal{A}$ with $(F(i), F(p))$ a conflation in $\mathcal{B}$, the given pair $(i, p)$ is necessarily a conflation.

    We prove  the statement above. Assume that $i\colon X\rightarrow Y$ and $p\colon Y\rightarrow Z$ are the given morphisms. By Lemma~\ref{lem:pte}(2), the conflation $(F(i), F(p))$ is isomorphic to the image of a conflation in $\mathcal{A}$. More precisely, there is a conflation $(i', p')$ in $\mathcal{A}$ with $i'$ starting at $X$ and $p'$ ending at $Z$, such that $(F(i'), F(p'))$ is isomorphic to $(F(i), F(p))$. By the equivalence $F$, this isomorphism implies that the composable pairs $(i', p')$ and $(i, p)$ are isomorphic. It follows that $(i, p)$ is also a conflation.
\end{proof}

The following consideration is inspired by \cite[Subsection~3A]{DI}; compare \cite{Chen12}. Let $\mathcal{A}$ be a Frobenius category. A full additive subcategory $\mathcal{F}\subseteq \mathcal{A}$ is called \emph{removable} if $\mathcal{F}\subseteq \mathcal{P}$ and the canonical functor $\pi_\mathcal{F}\colon \mathcal{A}\rightarrow \mathcal{A}/{[\mathcal{F}]}$ sends any conflation in $\mathcal{A}$ to a kernel-cokernel pair in $\mathcal{A}/{[\mathcal{F}]}$. Recall that $\mathcal{E}$ is the exact structure on $\mathcal{A}$. Denote by $\widetilde{\pi_\mathcal{F}(\mathcal{E})}$ the isomorphism closure of $\pi_\mathcal{F}(\mathcal{E})$.

\begin{lem}\label{lem:removable}
    Assume that $\mathcal{F}$ is a removable subcategory  of $\mathcal{A}$. Then $(\mathcal{A}/{[\mathcal{F}]}, \widetilde{\pi_\mathcal{F}(\mathcal{E})})$ becomes a Frobenius exact category, whose subcategory of projectives is precisely $\mathcal{P}/[\mathcal{F}]$.  Moreover, the canonical functor $\pi_\mathcal{F} \colon \mathcal{A}\rightarrow \mathcal{A}/{[\mathcal{F}]}$ is a pretriangle-equivalence between these two Frobenius categories.
\end{lem}

\begin{proof}
    This is a special case of \cite[Theorem~3.6]{DI}; compare \cite[Remark~3.3(c)]{DI}.
\end{proof}

The following terminology will be convenient.

\begin{defn}\label{defn:Fq}
    An additive functor $F\colon \mathcal{A}\rightarrow \mathcal{B}$ between two Frobenius categories is called a  \emph{Frobenius quotient},  provided that $F$ admits a factorization
    $$\mathcal{A} \stackrel{\pi_\mathcal{F}} \longrightarrow  \mathcal{A}/[\mathcal{F}] \stackrel{\bar{F}}\longrightarrow \mathcal{B} $$
    with $\mathcal{F}$ a removable subcategory of $\mathcal{A}$ and $\bar{F}$ an exact equivalence.
\end{defn}

In the situation above, $\mathcal{F}$ necessarily equals ${\rm  Ker}(F)$ and the factorization is indeed unique. By the very definition,  any Frobenius quotient is  essentially given by the canonical functor $\pi_\mathcal{F}\colon \mathcal{A}\rightarrow \mathcal{A}/{[\mathcal{F}]}$ for some removable subcategory $\mathcal{F}$. In particular, any Frobenius quotient is an exact functor.

The following result might be viewed as a recognition theorem for Frobenius quotients.

\begin{thm}\label{thm:Frob}
    Let $F\colon \mathcal{A}\rightarrow \mathcal{B}$ be an additive functor between two Frobenius categories. Then $F$ is a Frobenius quotient if and only if it is a pretriangle-equivalence,  objective, full and dense.
\end{thm}

\begin{proof}

The ``only if" part follows immediately from Lemma~\ref{lem:removable}. For the ``if" part, we assume that the given functor $F$ is a pretriangle-equivalence, which is objective, full and dense.

    The given functor $F$ admits the following canonical factorization.
    $$\mathcal{A}\stackrel{\pi}\longrightarrow \mathcal{A}/[{\rm Ker}(F)] \stackrel{\bar{F}}\longrightarrow \mathcal{B}$$
    Here, $\pi=\pi_{{\rm Ker}(F)}$ denotes the canonical functor. By Lemma~\ref{lem:pte}(1),  the essential kernel ${\rm Ker}(F)$ belongs to $\mathcal{P}$. By Lemma~\ref{lem:objective}, $\bar{F}$ is an equivalence of categories. Since $F$ is exact, it follows that $\pi$ sends conflations in $\mathcal{A}$ to kernel-cokernel pairs in $\mathcal{A}/[{\rm Ker}(F)]$. Thus, by the very definition, the subcategory ${\rm Ker}(F)$ is removable.

In view of Lemma~\ref{lem:removable}, we infer that $\bar{F}$ is exact and sends projectives in $\mathcal{A}/[{\rm Ker}(F)]$ to projectives in $\mathcal{B}$. Both $F$ and $\pi$ are pretriangle-equivalences. It follows from the factorization above that so is $\bar{F}$. We infer from Proposition~\ref{prop:pte} that $\bar{F}$ is an exact equivalence. Therefore, $F$ is a Frobenius quotient.
\end{proof}

The following well-known example shows the ubiquity of Frobenius quotients.

\begin{exm}
    {\rm Let $\mathcal{A}$ be any Frobenius category. An unbounded complex $P^\bullet=(P^n, d^n)$ of projective objects is acyclic if each differential $d^n\colon P^n\rightarrow P^{n+1}$ factors as
    $$P^n \stackrel{p^n}\longrightarrow Z^{n+1} \stackrel{i^{n+1}} \longrightarrow P^{n+1}$$
    and each pair $(i^n, p^n)$ is a conflation. The object $Z^{n+1}=Z^{n+1}(P^\bullet)$ is just the $(n+1)$-th cocycle of $P^\bullet$. Denote by $C_{\rm ac}(\mathcal{P})$ the category of such acyclic complexes. It becomes an exact category with conflations being chainwise-split short exact sequences. Moreover, it is a Frobenius category.

    It is well known that the zeroth-cocycle functor
    $$Z^0\colon C_{\rm ac}(\mathcal{P})\longrightarrow \mathcal{A}$$
    is a pretriangle-equivalence; see \cite[the proof of Theorem~4.3]{Kel94}. Moreover, it is full, dense and objective. To see that it is objective, we use the following observation: for any chain map $f^\bullet \colon P^\bullet \rightarrow Q^\bullet$ with $Z^0(f^\bullet)=0$, it necessarily factors through the mapping cone of the identity endomorphism on
   $$\cdots \rightarrow P^{-2}\stackrel{d^{-2}}\rightarrow P^{-1} \rightarrow 0\rightarrow P^1 \stackrel{d^1}\rightarrow P^2 \rightarrow \cdots.$$
    The essential kernel of $Z^0$ is given by contractible complexes with the $(-1)$-th differential being zero. We infer that $Z^0$ is a Frobenius quotient.}
\end{exm}

\section{The $n$-inflation categories}\label{sec:inf}

In this section, we study  the category of $n$-inflations. The main result states that any Frobenius quotient induces Frobenius quotients between the corresponding categories of $n$-inflations; see Theorem~\ref{thm:inf}.

Throughout this section, $(\mathcal{A}, \mathcal{E})$ is a Frobenius category and $n\geq 1$. By an \emph{$n$-inflation} in $\mathcal{A}$, we mean a sequence of $(n-1)$ composable inflations
$$X^1 \stackrel{\iota_X^1} \longrightarrow X^2 \stackrel{}\longrightarrow \cdots \longrightarrow X^{n-1}\stackrel{\iota_X^{n-1}}\longrightarrow X^n. $$
Such an $n$-inflation will be denoted by $(X^s;\iota_X^s)$. We mention that a $2$-inflation is just an inflation, and a $1$-inflation means a single object. A morphism
$$(f^s)\colon (X^s;\iota_X^s)\longrightarrow (Y^s;\iota_Y^s)$$
between two $n$-inflations is given by morphisms $f^s\colon X^s\rightarrow Y^s$ satisfying
$$\iota_Y^s\circ f^s=f^{s+1}\circ \iota_X^s$$
for all $1\leq s\leq n-1$. This forms the category ${\rm Inf}_n(\mathcal{A})$ of $n$-inflations. We mention its relation to the well-known Waldhausen S-construction; see \cite{War} and compare \cite{Arn, Bar, Sim, Zh}.
We write ${\rm Inf}_2(\mathcal{A})$ as ${\rm Inf}(\mathcal{A})$, and identify ${\rm Inf}_1(\mathcal{A})$ with $\mathcal{A}$ itself.

Each object $A$ in $\mathcal{A}$ gives rise to an $n$-inflation
$$\theta^1(A)=(A \stackrel{{\rm Id}_A}\longrightarrow A\longrightarrow \cdots \longrightarrow  A \stackrel{{\rm Id}_A}\longrightarrow A).$$
This gives rise to an additive functor $\theta^1\colon \mathcal{A}\rightarrow {\rm Inf}_n(\mathcal{A})$. More generally, for $1\leq s\leq n$, we have the following $n$-inflation
\begin{align}\label{equ:thetan}
    \theta^s(A)=(0\longrightarrow 0\longrightarrow \cdots \longrightarrow A \stackrel{{\rm Id}_A}\longrightarrow A\longrightarrow \cdots \longrightarrow  A \stackrel{{\rm Id}_A}\longrightarrow A),
\end{align}
which consists of $s-1$ copies of zero.

The category ${\rm Inf}_n(\mathcal{A})$ has a natural exact structure $\mathcal{E}_n$ in the following manner: a sequence
$$(X^s;\iota_X^s) \stackrel{(f^s)}\longrightarrow (Y^s;\iota_Y^s) \stackrel{(g^s)}\longrightarrow (Z^s;\iota_Z^s)$$
belongs to $\mathcal{E}_n$ if and only if each component $(f^s, g^s)$ is a conflation in $\mathcal{A}$. Moreover, $({\rm Inf}_n(\mathcal{A}), \mathcal{E}_n)$ is a Frobenius category, whose projective-injective objects are precisely  $\bigoplus_{s=1}^n\theta^s(P_s)$ for some projectives $P_s$ in $\mathcal{A}$; compare \cite{Chen11} and \cite[Section~5]{Sim15}.

In view of Lemma~\ref{lem:strictly-i}, the following observation is an immediate consequence of the definition.

\begin{lem}\label{lem:inf-ninf}
    Let $(f^s)\colon (X^s;\iota_X^s)\rightarrow (Y^s;\iota_Y^s)$ be a morphism between $n$-inflations. Then $(f^s)$ is an inflation in ${\rm Inf}_n(\mathcal{A})$ if and only if each commutative square
    $$\xymatrix{
 Y^s \ar[r]^-{\iota_Y^s} & Y^{s+1}\\
 X^s\ar[u]^-{f^s} \ar[r]^-{\iota_X^s} & X^{s+1} \ar[u]_-{f^{s+1}}
    }$$
    is strictly inflated. \hfill $\square$
\end{lem}

Let $m\geq 1$. We will also consider the category ${\rm Inf}_m({\rm Inf}_n(\mathcal{A}))$ of $m$-inflations in ${\rm Inf}_n(\mathcal{A})$. In view of Lemma~\ref{lem:inf-ninf}, an object in ${\rm Inf}_m({\rm Inf}_n(\mathcal{A}))$ is given by a grid in $\mathcal{A}$,
\[\xymatrix@R=13pt{
X^{1, m}   \ar[r] & X^{2,m}  \ar[r] & \cdots \ar[r] & X^{n-1, m}\ar[r] & X^{n, m}\\
X^{1, m-1} \ar[u] \ar[r] & X^{2,m-1}  \ar[u] \ar[r] & \cdots \ar[r] & X^{n-1, m-1}\ar[u] \ar[r] & X^{n, m-1}\ar[u] \\
\vdots\ar[u]    & \vdots \ar[u]  & \cdots  & \vdots\ar[u]   & \vdots  \ar[u]\\
X^{1, 2} \ar[u] \ar[r] & X^{2,2} \ar[u]  \ar[r] & \cdots \ar[r] & X^{n-1, 2} \ar[u]\ar[r] & X^{n, 2} \ar[u]\\
X^{1, 1} \ar[u]\ar[r] & X^{2,1} \ar[u] \ar[r] & \cdots \ar[r] & X^{n-1, 1} \ar[u]\ar[r] & X^{n, 1} \ar[u]
}\]
 where each small commutative square is strictly inflated. The conflations are given componentwise. By transpose, we infer the following symmetry
 $${\rm Inf}_m({\rm Inf}_n(\mathcal{A}))={\rm Inf}_n({\rm Inf}_m(\mathcal{A})). $$
 This common Frobenius category will be denoted by ${\rm Inf}_{m, n}(\mathcal{A})$, which might be called the category of \emph{$(m, n)$-inflations} in $\mathcal{A}$; compare \cite{Sim15, LO}. 

 \begin{exm}\label{exm:grid}
     {\rm Let $\mathbb{K}$ be a field. We will consider the truncated polynomial algebra $\mathbb{K}[t]/{(t^p)}$ for some $p\geq 2$, which is naturally $\mathbb{Z}$-graded. Consider the Frobenius abelian category  ${\rm mod}^\mathbb{Z}\mbox{-}\mathbb{K}[t]/{(t^p)}$ of finite- dimensional $\mathbb{Z}$-graded modules.  For $n\geq 1$, we write
$$\mathcal{S}^\mathbb{Z}_n(p)={\rm Inf}_n({\rm mod}^\mathbb{Z}\mbox{-}\mathbb{K}[t]/{(t^p)}).$$
Furthermore, $\mathcal{S}^\mathbb{Z}(p)=\mathcal{S}^\mathbb{Z}_2(p)$ is called the \emph{graded submodule category} \cite{RS08}. We mention the comparison of the ungraded submodule category with the original Birkhoff problem in \cite{GKKP}. 

More generally, we write 
$$\mathcal{S}^\mathbb{Z}_{m,n}(p)={\rm Inf}_{m, n}({\rm mod}^\mathbb{Z}\mbox{-}\mathbb{K}[t]/{(t^p)}).$$
Its objects are given by certain $m\times n$-grids of monomorphisms in  ${\rm mod}^\mathbb{Z}\mbox{-}\mathbb{K}[t]/{(t^p)}$. In view of Remark~\ref{rem:sinf}, the grid above is completely determined by the upper boundary and the rightmost boundary. Consequently, the category $\mathcal{S}^\mathbb{Z}_{m,n}(p)$ is equivalent to the category consisting of $2$-flags of monomorphisms in ${\rm mod}^\mathbb{Z}\mbox{-}\mathbb{K}[t]/{(t^p)}$; see \cite[p.198]{KLM1} and \cite[Section~4]{Sim15}. }
 \end{exm}

Assume that $n\geq 2$. Recall the functor $\theta^1\colon \mathcal{A}\rightarrow {\rm Inf}_n(\mathcal{A})$ above. We have the projection functor
$${\rm pr}^1\colon {\rm Inf}_n(\mathcal{A}) \longrightarrow \mathcal{A},$$
which sends $(X^s; \iota_X^s)$ to the first component $X^1$. Similarly,  the projection functor
$${\rm pr}^{[2, n]}\colon {\rm Inf}_n(\mathcal{A}) \rightarrow  {\rm Inf}_{n-1}(\mathcal{A})$$
sends $(X^s; \iota_X^s)$ to the  $(n-1)$-inflation starting at $X^2$ and ending at $X^n$. On the other hand, we have the \emph{extension-by-zero} functor
$${\rm Ez}\colon {\rm Inf}_{n-1}(\mathcal{A}) \longrightarrow  {\rm Inf}_{n}(\mathcal{A}), $$
which extends an $(n-1)$-inflation $(U^s; \iota_U^s)$ to an $n$-inflation by adding the zero morphism $0\rightarrow U^1$ to the left.

For any $n$-inflation $(X^s; \iota_X^s)$, we denote by $C^s$ the cokernel of $$\iota_X^{s}\circ \cdots \circ \iota_X^1\colon X^1\longrightarrow X^{s+1}. $$ Then we have the induced inflation $\iota_C^s \colon C^s\rightarrow C^{s+1}$ by $\iota_X^{s+1}$. In summary, we have an $(n-1)$-inflation
$$ {\rm Cok}(X^s; \iota_X^s)=(C^s; \iota_C^s).$$
This gives rise to the \emph{cokernel} functor
$${\rm Cok}\colon {\rm Inf}_n(\mathcal{A}) \longrightarrow  {\rm Inf}_{n-1}(\mathcal{A}).$$
The five functors above are all exact and send projectives to projectives. They induce the corresponding functors between the stable categories.

In contrast, the following \emph{embedding-into-injective} functor
$${\rm Ei}\colon \underline{\mathcal{A}}\longrightarrow \underline{\rm Inf}_n(\mathcal{A})$$ is only defined between the stable categories. For each object $A$ in $\mathcal{A}$, ${\rm Ei}(A)$ is given by the following $n$-inflation
$$A \stackrel{i_A} \longrightarrow I(A) \stackrel{{\rm Id}_{I(A)}} \longrightarrow I(A) \longrightarrow \cdots  \longrightarrow I(A) \stackrel{{\rm Id}_{I(A)}} \longrightarrow I(A),$$
where $I(A)$ is injective and $i_A$ is an inflation. For each morphism $\underline{f}\colon A\rightarrow A'$ in $\underline{\mathcal{A}}$, we fix any morphism $I(f)\colon I(A)\rightarrow I(A')$ satisfying $I(f)\circ i_A=i_{A'}\circ f$. Then ${\rm Ei}(\underline{f})$ is represented by the morphism $(f, I(f), \cdots, I(f))$ in ${\rm Inf}_n(\mathcal{A})$.

We refer to \cite[1.4]{BBD} for recollements. The following fact is known to experts.

\begin{lem}
Keep the notation above. Then we have a recollement of triangulated categories.
\[\xymatrix{
\underline{\rm Inf}_{n-1}(\mathcal{A}) \ar[rr]^-{\rm Ez} &&  \underline{\rm Inf}_{n}(\mathcal{A}) \ar@/_1.7pc/[ll]|{\rm Cok}  \ar@/^1.7pc/[ll]|{{\rm pr}^{[2, n]}}  \ar[rr]^{{\rm pr}^1} && \underline{\mathcal{A}} \ar@/_1.7pc/[ll]|{\theta^1}  \ar@/^1.7pc/[ll]|{\rm Ei}
}\]
\end{lem}

\begin{proof}
    We mention that adjoint pairs $({\rm Cok}, {\rm Ez})$, $({\rm Ez}, {\rm pr}^{[2, n]})$ and $(\theta^1, {\rm pr}^1)$ exist already among the unstable categories. Then we have the induced ones; see \cite[Lemma~2.3]{Chen12-Ld}. To see that ${\rm Ei}$ is a triangle functor, we use the fact that the adjoint functor of any triangle functor is also triangle; see \cite[Lemma~8.3]{Kel96}.
\end{proof}

\begin{prop}\label{prop:inf}
    Let $F\colon \mathcal{A} \rightarrow \mathcal{B}$ be a pretriangle-equivalence between two Frobenius categories. Then for each $n\geq 1$, the induced functor ${\rm Inf}_n(F)\colon {\rm Inf}_n(\mathcal{A}) \rightarrow {\rm Inf}_n(\mathcal{B})$ is also a pretriangle-equivalence.
\end{prop}

\begin{proof}
    The induced functor ${\rm Inf}_n(F)$ above is clearly exact and sends projectives to projectives. It suffices to show that it induces a stable equivalence.  The key observation is the following commutative diagram between recollements.
    \[\xymatrix{
\underline{\rm Inf}_{n-1}(\mathcal{A}) \ar[dd]_-{{\rm Inf}_{n-1}(F)} \ar[rr]^-{\rm Ez} &&  \underline{\rm Inf}_{n}(\mathcal{A}) \ar[dd]^-{{\rm Inf}_{n}(F)} \ar@/_1.7pc/[ll]|{\rm Cok}  \ar@/^1.7pc/[ll]|{{\rm pr}^{[2, n]}}  \ar[rr]^{{\rm pr}^1} && \underline{\mathcal{A}}  \ar[dd]^-{F} \ar@/_1.7pc/[ll]|{\theta^1}  \ar@/^1.7pc/[ll]|{\rm Ei}\\
\\
\underline{\rm Inf}_{n-1}(\mathcal{B}) \ar[rr]^-{\rm Ez} &&  \underline{\rm Inf}_{n}(\mathcal{B}) \ar@/_1.7pc/[ll]|{\rm Cok}  \ar@/^1.7pc/[ll]|{{\rm pr}^{[2, n]}}  \ar[rr]^{{\rm pr}^1} && \underline{\mathcal{B}} \ar@/_1.7pc/[ll]|{\theta^1}  \ar@/^1.7pc/[ll]|{\rm Ei}
}\]
The required stable equivalence follows immediately from the comparison theorem \cite[Theorem~2.5]{PS} for recollements and by induction.
\end{proof}

The following main result indicates the compatibility between inflation categories and Frobenius quotients.

\begin{thm}\label{thm:inf}
    Let $F\colon \mathcal{A} \rightarrow \mathcal{B}$ be a Frobenius quotient between two Frobenius categories. Then for each $n\geq 1$, the induced functor ${\rm Inf}_n(F)\colon {\rm Inf}_n(\mathcal{A}) \rightarrow {\rm Inf}_n(\mathcal{B})$ is also a Frobenius quotient, whose  essential kernel equals ${\rm add}\; \{\theta^j(P)\; |\; 1\leq j\leq n, P\in {\rm Ker}(F)\}$. 
\end{thm}

Here, we denote by ${\rm add}$ the smallest additive subcategory closed under taking direct summands. 

\begin{proof}
   Denote by $\mathcal{P}$ the full subcategory of projectives in  $\mathcal{A}$. Set $\mathcal{F}={\rm Ker}(F)$. We have $\mathcal{F}\subseteq \mathcal{P}$. We observe that ${\rm Inf}_n(F)$ sends an $n$-inflation  $(X^s; \iota_X^s)$  to the $n$-inflation $(F(X^s); F(\iota_X^s))$. In view of Theorem~\ref{thm:Frob} and Proposition~\ref{prop:inf}, it suffices to show that ${\rm Inf}_n(F)$ is dense, full and objective.

    We first show that ${\rm Inf}_n(F)$ is dense. Let $(U^s; \iota_U^s)$ be an object in ${\rm Inf}_n(\mathcal{B})$. Since $F$ is dense, we might assume that $U^1=F(X^1)$ for some object $X^1$ in $\mathcal{A}$. Consider the following conflation
    $$F(X^1)=U^1\stackrel{\iota_U^1}\longrightarrow U^2 \stackrel{}\longrightarrow C,$$
    with $C$ the cokernel of $\iota_U^1$. We may assume that $C=F(A)$. The conflation above corresponds to an element in ${\rm Ext}^1_\mathcal{B}(F(A), F(X^1))$. By the isomorphism in Lemma~\ref{lem:pte}(2), we have a conflation
    $$X^1\stackrel{\iota_X^1}\longrightarrow X^2 \stackrel{}\longrightarrow A $$
    in $\mathcal{A}$, which is sent by $F$ to the conflation above. In particular, up to isomorphism, we might replace $U^2$ with $F(X^2)$, and $\iota_U^1$ with $F(\iota_X^1)$. We continue this process and obtain an $n$-inflation $(X^s; \iota_X^s)$ in $\mathcal{A}$, which is sent by ${\rm Inf}_n(F)$ to the given object $(U^s; \iota_U^s)$.

    For the fullness of ${\rm Inf}_n(F)$, we take two $n$-inflations   $(X^s; \iota_X^s)$  and  $(Y^s; \iota_Y^s)$ in $\mathcal{A}$. Assume that $$(g^s)\colon (F(X^s); F(\iota_X^s))\longrightarrow (F(Y^s); F(\iota_Y^s))$$
        is a morphism in ${\rm Inf}_n(\mathcal{B})$. Since $F$ is full, we take a morphism $f^1\colon X^1\rightarrow Y^1$ in $\mathcal{A}$ satisfying $F(f^1)=g^1$. For the same reason, there is a morphism $f\colon X^2\rightarrow Y^2$ satisfying $F(f)=g^2$. By $F(\iota_Y^1)\circ g^1=g^2\circ F(\iota_X^1)$, we infer that $F$ annihilates $\iota_Y^1\circ f^1-f\circ \iota_X^1$. Since $F$ is objective, there exist two morphisms $a\colon X^1\rightarrow P$ and $b\colon P \rightarrow Y^2$ with $P\in \mathcal{F}$, which satisfy
    $$\iota_Y^1\circ f^1-f\circ \iota_X^1=b\circ a. $$
Since $P$ is injective and $\iota_X^1$ is an inflation, we have a morphism $a'\colon X^2\rightarrow P$ with $a=a'\circ \iota_X^1$. Set $f^2=f+b\circ a'$. One verifies that $\iota_Y^1\circ f^1=f^2\circ \iota_X^1$ and $F(f^2)=g^2$. We might repeat the argument to obtain  suitable morphisms $f^s\colon {X^s }\rightarrow Y^s$, which consist of a required morphism $(f^s)$ in ${\rm Inf}_n(\mathcal{A})$.

It remains to show that ${\rm Inf}_n(F)$ is objective. The proof is more subtle.  For this, we take any morphism
$$(f^s)\colon (X^s; \iota_X^s) \longrightarrow  (Y^s; \iota_Y^s) $$
in ${\rm Inf}_n(\mathcal{A})$, which is annihilated by ${\rm Inf}_n(F)$. In particular, we have $F(f^s)=0$ for each $1\leq s \leq n$. Since $F$ is objective, there exist two morphisms $a^1\colon X^1 \rightarrow P^1$ and $b^1\colon P^1\rightarrow Y^1$ with $P^1\in \mathcal{F}$, which satisfy $f^1=b^1\circ a^1$. Since $P^1$ is injective, for all $s\geq 2$, there are morphisms $a^s\colon X^s\rightarrow P^1$ satisfying $a^s\circ \iota_X^{s-1}=a^{s-1}$. These morphisms form a morphism
$$(a^s)\colon (X^s; \iota_X^s)\longrightarrow \theta^1(P^1).$$
The morphism $b^1$ induces a unique morphism
$$(b^s)\colon \theta^1(P^1)\longrightarrow (Y^s; \iota_Y^s).$$
Here, we refer to (\ref{equ:thetan}) for $\theta^1(P^1)$. We now replace the given morphism $(f^s)$ by
$$(f^s)-(b^s)\circ (a^s).$$
Consequently, we will assume that $f^1=0$.

Consider the cokernel $p\colon X^2\rightarrow C$ of $\iota_X^1$. Using the assumption that $f^1=0$, we deduce that there is a unique morphism $\bar{f^2}\colon C\rightarrow Y^2$ satisfying $f^2=\bar{f^2}\circ p$. Therefore, we have
$$0=F(f^2)=F(\bar{f^2})\circ F(p).$$
Since $F(p)$ is a deflation, we infer that $F(\bar{f^2})=0$. Since $F$ is objective, there is an object $P^2\in \mathcal{F}$ such that $\bar{f^2}$ factors as $C\stackrel{c}\rightarrow P^2\stackrel{h^2}\rightarrow Y^2$. Set $e^1=0$ and $e^2=c\circ p$. By the injectivity of $P^2$, there are morphisms $e^s\colon X^s\rightarrow P^2$ satisfying $e^s\circ \iota_X^{s-1}=e^{s-1}$ for all $s\geq 3$. We form a morphism
$$(e^s)\colon (X^s; \iota_X^s)\longrightarrow \theta^2(P^2).$$
The morphism $h^2$ induces a unique morphism
$$(h^s)\colon \theta^2(P^2)\longrightarrow (Y^s; \iota_Y^s),$$
where $h^1$ is taken to be the zero morphism. We now replace $(f^s)$ by the new morphism
$$(f^s)-(h^s)\circ (e^s).$$
Consequently, we may assume that $f^1=f^2=0$. We continue this process and prove that the given morphism does factor through $\bigoplus_{1\leq j\leq n}\theta^j(P^j)$ for some $P^j\in \mathcal{F}$.  Indeed, the essential kernel of ${\rm Inf}_n(F)$ equals ${\rm Inf}_n(\mathcal{F})$. Consequently, the functor ${\rm Inf}_n(F)$ is objective. This completes the whole proof.
\end{proof}

\section{Gorenstein-projective modules and projective-module factorizations}\label{sec:PMF}

In this section, we relate the categories of graded Gorenstein-projective modules over two different graded rings. The main tool is the category of projective-module factorizations in \cite{BHU, Tri, SZ}.

Let $H$ be a group, which is written multiplicatively. Let $R=\oplus_{h\in H}R_h$ be an $H$-graded ring. Denote by ${\rm Mod}^H\mbox{-}R$ the abelian category of $H$-graded right $R$-modules, whose morphisms respect both the $R$-actions and gradings.  The full subcategory formed by graded projective modules is denoted by ${\rm Proj}^H\mbox{-}R$. For each $g\in H$ and a graded module $M=\oplus_{h\in H} M_h$, its \emph{shifted} module $M(g)$ is graded such that $M(g)_h=M_{gh}$. This gives rise to the \emph{degree-shift automorphism} $(g)$ on ${\rm Mod}^H\mbox{-}R$. We refer to \cite{NVO} for graded rings and modules.

Let $\sigma$ be a graded ring automorphism on $R$. For a graded $R$-module $M$, the \emph{twisted} module $M^\sigma$ is defined as follows. Its typical element of degree $h$ is given by $m^\sigma$ with $m\in M_h$. The $R$-action is given by $m^\sigma a=(m\sigma(a))^\sigma$ for $a\in R$. We obtain the \emph{twisting autoequivalence} $(-)^\sigma$ on ${\rm Mod}^H\mbox{-}R$.

Recall that a graded $R$-module $M$ is \emph{Gorenstein-projective} \cite{ABr69, EJ} if there exists a totally acyclic complex $P^\bullet$ of graded projective modules with $Z^{0}(P^\bullet)\simeq M$. Denote by ${\rm GProj}^H\mbox{-}R$ the full subcategory containing such modules. We mention that ${\rm Proj}^H\mbox{-}R\subseteq {\rm GProj}^H\mbox{-}R$. Since ${\rm GProj}^H\mbox{-}R$ is closed under extensions in ${\rm Mod}^H\mbox{-}R$, it becomes an exact category with conflations given by short exact sequences of Gorenstein-projective modules. Moreover, by \cite[Proposition~3.8]{Bel05},  it is a Frobenius category, whose projective-injectives are precisely graded projective modules. Denote by $\underline{\rm GProj}^H\mbox{-}R$ its stable category.

Let $S$ be an $H$-graded ring. Fix a homogeneous element $\omega\in S$ of degree $\vec{c}$, such that $\vec{c}$ is a central element in $H$. Assume  that $\omega$ is \emph{regular}, that is, a non-zero-divisor on both sides, and \emph{normal}, that is, $S\omega=\omega S$. It follows that there is a unique graded automorphism $\sigma$ on $S$ satisfying
$$a\omega=\omega \sigma(a).$$
We observe that $\sigma(\omega)=\omega$. The quadruple $(S, \omega, \vec{c}, \sigma)$ will be called an \emph{$H$-graded nc-quadruple}.

Let $M$ be a graded $S$-module. The element $\omega$ induces the following natural homomorphism in ${\rm Mod}^H\mbox{-}S$
 $$\omega_M\colon M\longrightarrow M^\sigma(\vec{c}), \; m\mapsto (m\omega)^\sigma.$$
 Following \cite{BHU, Tri, SZ}, an \emph{$H$-graded module $(n+1)$-factorization} of $\omega$ means the following sequence in ${\rm Mod}^H\mbox{-}S$
 $$ M^0\stackrel{d_M^0} \longrightarrow M^1 \stackrel{}\longrightarrow \cdots \longrightarrow M^{n-1}
 \stackrel{d_M^{n-1}} \longrightarrow  M^n \stackrel{d_M^n}\longrightarrow (M^0)^\sigma(\vec{c}), $$
subject to the relations
\begin{align*}
   &  d_M^n\circ \cdots \circ d_M^1\circ d_M^0=\omega_{M^0}, \;  (d_M^0)^\sigma(\vec{c})\circ d_M^n\circ \cdots \circ d_M^2\circ d_M^1=\omega_{M^1}, \\
    & \cdots, \mbox{ and } (d_M^{n-1}\circ \cdots \circ d_M^0)^\sigma(\vec{c})\circ d_M^n=\omega_{M^n}.
\end{align*}
Such an $(n+1)$-factorization will be denoted by $(M^t, d_M^t)$. Denote by ${\rm F}_{n+1}^H(S; \omega)$ the category of $H$-graded module $(n+1)$-factorizations of $\omega$. We mention that ${\rm F}_{n+1}^H(S; \omega)$  is equivalent to the graded module category over a certain matrix ring; see \cite[Proposition~3.1]{SZ}.

An  $(n+1)$-factorization $(M^t, d_M^t)$ is called a \emph{projective-module $(n+1)$-factorization} provided each component $M^i$ is projective. Denote by ${\rm PF}_{n+1}^H(S; \omega)$ the category of $H$-graded projective-module $(n+1)$-factorizations of $\omega$. Since it is closed under extensions  in ${\rm F}_{n+1}^H(S; \omega)$, it becomes an exact category. Moreover, it is a Frobenius category, whose projective-injectives are given by the \emph{trivial $(n+1)$-factorizations} explained below.

\begin{rem}
    Since $\omega$ is regular, $\omega_P$ is monic  for any graded projective $S$-module $P$. It follows that in any  projective-module $(n+1)$-factorization $(P^t, d_P^t)$ of $\omega$, each differential $d_P^t$ is monic.
\end{rem}

Each graded projective $S$-module $P$ yields the following $(n+1)$-factorization
\begin{align}\label{equ:psi}
\psi^0(P)=(P\stackrel{{\rm Id}_P}\longrightarrow P \stackrel{}\longrightarrow \cdots \longrightarrow P \stackrel{{\rm Id}_P}\longrightarrow P \stackrel{\omega_P} \longrightarrow P^\sigma(\vec{c})).
\end{align}
More generally, for $1\leq i\leq n$, we denote by $\psi^i(P)$ the following $(n+1)$-factorization
$$P'\stackrel{{\rm Id}_{P'}}\longrightarrow P' \stackrel{}\longrightarrow \cdots \longrightarrow P'\stackrel{\omega_{P'}} \longrightarrow P \stackrel{{\rm Id}_P}\longrightarrow P \longrightarrow \cdots \longrightarrow P $$
with $P'=(P)^{\sigma^{-1}}(\vec{c}^{-1})$, where we identify $P$ with $(P')^\sigma(\vec{c})$ and $\omega_{P'}$ is at the $i$-th position from the left.  By definition, a trivial $(n+1)$-factorization is a direct summand of  $\bigoplus_{0\leq i\leq n}\psi^i(P_i)$ for some graded projective $S$-modules $P_i$.

Consider the $H$-graded quotient ring $R=S/(\omega)$. Denote by ${\rm GProj}^H_{\rm fpd}\mbox{-}R$ the full subcategory of ${\rm GProj}^H\mbox{-}R$ formed by those graded $R$-modules which have finite projective dimension as $S$-modules. It inherits the exact structure from ${\rm GProj}^H\mbox{-}R$ and becomes a Frobenius category.  Here, we use the fact that any projective $R$-module has projective dimension at most one as an $S$-module.

Let $(P^t, d_P^t)$ be a graded projective-module $(n+1)$-factorization. Set $C^s$ to be the cokernel of the monomorphism
$$ d_P^{s-1}\circ \cdots \circ d_P^0 \colon P^0\longrightarrow P^s.$$
Since $\omega$ vanishes on $C^s$, it is a graded $R$-module, which has finite projective dimension as an $S$-module. By a graded version of \cite[Theorem~2.10]{Chen24}, it is a graded Gorenstein-projective $R$-module. In other words, we have that $C^s$ belongs to ${\rm GProj}^H_{\rm fpd}\mbox{-}R$. For $1\leq s\leq n-1$,   the monomorphism $d_P^s\colon P^s\rightarrow P^{s+1}$ induces a monomorphism $\iota_C^s\colon C^s \rightarrow C^{s+1}$. We observe that the cokernel of $\iota_C^s$ is identified with the cokernel of $d_P^s$, and belongs to ${\rm GProj}^H_{\rm fpd}\mbox{-}R$. In other words, $\iota_C^s$ is an inflation in ${\rm GProj}^H_{\rm fpd}\mbox{-}R$. Therefore, we obtain an $n$-inflation $(C^s; \iota_C^s)$, which is denoted by ${\rm Cok}(P^t, d_P^t)$. This gives rise to the \emph{cokernel} functor
$${\rm Cok}\colon {\rm PF}_{n+1}^H(S; \omega)\longrightarrow {\rm Inf}_n({\rm GProj}^H_{\rm fpd}\mbox{-}R).$$

The following result is due to \cite[Section~4]{SZ}; compare \cite[Theorem~5.6]{Chen24}.

\begin{thm}\label{thm:SZ}
    Keep the assumptions above. Then the cokernel functor above is a Frobenius quotient, whose essential kernel is given by ${\rm Add}\; \{\psi^0(S(h))\; |\; h\in H\}$.
\end{thm}

Here, we denote by Add the smallest additive subcategory that is closed under taking direct summands and infinite coproducts.  

\begin{proof}
  By a graded version of \cite[Theorems~4.6~and~4.7]{SZ},  the cokernel functor above is a pretriangle-equivalence. It is full and dense by \cite[Lemma~4.3]{SZ}, and objective by \cite[Corollary~4.5]{SZ}. For the essential kernel, we just observe that ${\rm Cok}(P^t, d_P^t)=0$ if and only if $(P^t, d_P^t)$ is isomorphic to the trivial factorization $\psi^0(P)$ in (\ref{equ:psi}) for some graded projective $S$-module $P$.
\end{proof}

In what follows, we fix another graded automorphism $\tau$ on $S$ which satisfies $\tau(\omega)=\omega$ and $\tau^{-(n+1)}=\sigma$. Consider the skew polynomial ring $S[x; \tau]$ which satisfies $xa=\tau(a)x$ for $a\in S$. The \emph{noncommutative $(n+1)$-branched cover} \cite{Kno, MU21} of $S$ is defined to be the quotient ring
$$R_1=S[x; \tau]/{(x^{n+1}+\omega)}.$$
Consider the group $H_1$ which is obtained from $H$ by adding a new generator $\vec{x}$ subject to the relations
$$h\vec{x}=\vec{x}h \mbox{ for all } h\in H,  \mbox{ and } \vec{x}^{n+1}=\vec{c}.$$
We have a disjoint union
$$H_1=H \cup H\vec{x}\cup \cdots \cup H\vec{x}^{n}.$$
The ring $R_1$ is naturally $H_1$-graded such that ${\rm deg}(x)=\vec{x}$.

For each $(n+1)$-factorization $(M^t, d_M^t)$ in ${\rm F}^H_{n+1}(S; \omega)$, we associate an $H_1$-graded $R_1$-module $\Theta(M^t, d_M^t)=N$ as follows. For each $0\leq t\leq n$, we have
$$N_{H\vec{x}^t}=(M^t)^{\tau^t},$$
that is, a typical element in $N_{h\vec{x}^t}$ given by $(u^t)^{\tau^t}$ for $u^t\in M^t_h$. Therefore,  the $S$-action is given by
$$(u^t)^{\tau^t} a=(u^t\tau^t(a))^{\tau^t}, \mbox{ for all } a\in S.$$
The action of $x$ on $N$ is defined such that
$$(u^t)^{\tau^t} x=(d_M^t(u^t))^{\tau^{t+1}}$$
for $0\leq t\leq n-1$; and $(u^n)^{\tau^n}x=-u^0$ if $d_M^n(u^n)=(u^0)^\sigma\in (M^0)^\sigma(\vec{c})$. This defines a functor
$$\Theta\colon {\rm F}^H_{n+1}(S; \omega)\longrightarrow {\rm Mod}^{H_1}\mbox{-}R_1.$$

The following well-known fact essentially goes back to \cite[Proposition~2.1]{Kno}.

\begin{lem}\label{lem:Kno}
    Let $R_1=S[x; \tau]/{(x^{n+1}+\omega)}$. Then the functor $\Theta$ is an equivalence, which restricts to an exact equivalence
    $$\Theta\colon {\rm PF}^H_{n+1}(S; \omega)\stackrel{\sim}\longrightarrow {\rm GProj}_{\rm fpd}^{H_1}\mbox{-}R_1.$$
\end{lem}

\begin{proof}
    For the whole equivalence, a modified proof of \cite[Lemma~4.2]{Chen-Wu} works here.  For the restricted equivalence, the same argument in the proof of \cite[Proposition~4.4]{Chen-Wu} applies well. Here, we implicitly use the fact that an $R_1$-module has finite projective dimension if and only if so does the underlying $S$-module; see \cite[Lemma~4.1]{Chen-Wu}. We omit the details.
\end{proof}

The following result will be crucial for us.

\begin{prop}\label{prop:GProj}
  Recall that $R=S/(\omega)$ and $R_1=S[x; \tau]/{(x^{n+1}+\omega)}$. Then we have a Frobenius quotient
  $${\rm Cok}\circ \Theta^{-1}\colon  {\rm GProj}_{\rm fpd}^{H_1}\mbox{-}R_1\longrightarrow  {\rm Inf}_n({\rm GProj}_{\rm fpd}^{H}\mbox{-}R),$$
  whose essential kernel is given by ${\rm Add}\; \{R_1(h)\; |\; h\in H\}$.
\end{prop}

\begin{proof}
   By combining the exact equivalence in Lemma~\ref{lem:Kno} and the Frobenius quotient in Theorem~\ref{thm:SZ}, we obtain the required Frobenius quotient. For the essential kernel,  we just observe that $\Theta(\psi^0(S(h)))\simeq R_1(h)$ for any $h\in H$.
\end{proof}

We have the following immediate consequence.

\begin{cor}\label{cor:fgl}
    Assume further that the $H$-graded ring $S$ is right graded-noetherian having finite right graded global dimension. Then  we have a Frobenius quotient
  $${\rm Cok}\circ \Theta^{-1}\colon  {\rm GProj}^{H_1}\mbox{-}R_1\longrightarrow  {\rm Inf}_n({\rm GProj}^{H}\mbox{-}R),$$
  which restricts to a Frobenius quotient
  $${\rm Cok}\circ \Theta^{-1}\colon  {\rm Gproj}^{H_1}\mbox{-}R_1\longrightarrow {\rm Inf}_n({\rm Gproj}^{H}\mbox{-}R).$$
  The essential kernel of the restricted functor equals ${\rm add}\; \{R_1(h)\; |\; h\in H\}$. \hfill $\square$
\end{cor}

Here, we denote by ${\rm Gproj}^{H_1}\mbox{-}R_1$ and ${\rm Gproj}^{H}\mbox{-}R$ the full subcategory formed by  finitely generated graded Gorenstein-projective modules.

Let $\mathbb{K}$ be a field. Following Example~\ref{exm:grid}, we write $\mathcal{S}^\mathbb{Z}_n(p)={\rm Inf}_n({\rm mod}^\mathbb{Z}\mbox{-}\mathbb{K}[t]/{(t^p)})$.

\begin{exm}\label{exm:pq}
{\rm  Assume that $p, q\geq 2$. Denote by $L=L(p, q)$ the additive group generated by $\vec{x}$ and $\vec{y}$ subject to the relation $p\vec{x}=q\vec{y}$; this common value is denoted by $\vec{c}$. Consider the quotient algebra 
$A=\mathbb{K}[x, y]/(x^p+y^q)$ of the polynomial algebra in two variables, which is naturally $L$-graded by means of ${\rm deg}(x)=\vec{x}$ and ${\rm deg}(y)=\vec{y}$.  Since $A$ is $1$-Gorenstein, we infer that ${\rm Gproj}^L\mbox{-A}$ coincides with ${\rm MCM}^L(A)$, the category of \emph{$L$-graded maximal Cohen-Macaulay $A$-modules} \cite{AR}. 

 We claim that there are two Frobenius quotients
 $$ \mathcal{S}^\mathbb{Z}_{p-1}(q) \longleftarrow {\rm MCM}^L(A) \longrightarrow \mathcal{S}^\mathbb{Z}_{q-1}(p). $$
 Consequently, we have a stable equivalence
  $$ \underline{\mathcal{S}}^\mathbb{Z}_{p-1}(q) \simeq \underline{\mathcal{S}}^\mathbb{Z}_{q-1}(p),$$
  which is closely related to  the \emph{Happel-Seidel symmetry}; compare \cite{HS} and \cite[Theorem~6.11]{KLM1}.

For the claim, we consider the subgroup $H$ of $L$ generated by $\vec{y}$. Set $S=\mathbb{K}[y]$ and $\omega=y^q$. Then $S$ is naturally $H$-graded and ${\rm deg}(\omega)=\vec{c}$. We identify $A$ with $S[x]/(x^p+\omega)$, and $L$ with $H_1$. Then the Frobenius quotient
$$ {\rm MCM}^L(A) \longrightarrow \mathcal{S}^\mathbb{Z}_{p-1}(q)$$
follows from Corollary~\ref{cor:fgl}. For another one, we exchange the role of $x$ and $y$. We mention that graded maximal Cohen-Macaulay $A$-modules are also studied in \cite[Section~6]{JKS}.}
\end{exm}

\section{The main result}\label{sec:main}

 Let $H$ be a group. We fix an $H$-graded nc-quadruple $(S, \omega, \vec{c}, \sigma)$ and set $R=S/(\omega)$ to be the $H$-graded quotient ring. Fix $n, m\geq 1$. Moreover, we assume that $\tau$ and $\delta$ are two graded automorphisms on $S$ subject to the conditions:
 $$\tau\circ \delta =\delta \circ \tau, \tau^{-(n+1)}=\sigma=\delta^{-(m+1)}, \mbox{ and } \tau(\omega)=\omega=\delta(\omega). $$
 Fix an invertible central element $\lambda_0\in S$, which is homogeneous of degree $1_H$  and satisfies
 $$\tau(\lambda_0)=\lambda_0=\delta(\lambda_0), \mbox{ and } (\lambda_0)^{n+1}=1_S=(\lambda_0)^{m+1}.$$
 For example, we may take $\lambda_0=1_S$.

 Denote by $H_2$ the group obtained from $H$ by adjoining two new generators $\vec{x}$ and $\vec{y}$, which are subject to the relations
 $$\vec{x}\vec{y}=\vec{y}\vec{x},\;  \vec{x}^{n+1}=\vec{c}=\vec{y}^{m+1}, \mbox{ and } h\vec{x}=\vec{x}h, h\vec{y}=\vec{y}h \mbox{ for all } h\in H.$$
 Denote by $S_2=S[x; \tau][y; \delta]$ the skew polynomial ring with two variables. In particular, we have
 $$ yx=\lambda_0 xy, \; xa=\tau(a)x, \mbox{ and } ya=\delta(a)y \mbox{ for all } a\in S. $$
 Then $S_2$ is naturally $H_2$-graded by means of ${\rm deg}(x)=\vec{x}$ and ${\rm deg}(y)=\vec{y}$. We observe that $x^{n+1}+y^{m+1}+\omega$ is a homogeneous element in $S_2$. Consider the $H_2$-graded quotient ring
 $$R_2=S_2/(x^{n+1}+y^{m+1}+\omega).$$

 Recall that ${\rm GProj}^H_{\rm fpd}\mbox{-}R$ denotes the full subcategory of ${\rm GProj}^H\mbox{-}R$ formed by those graded Gorenstein-projective $R$-modules whose underlying graded $S$-modules have finite projective dimension. Denote by ${\rm Inf}_{m, n}({\rm GProj}^H_{\rm fpd}\mbox{-}R)$ the category of $(m, n)$-inflations in ${\rm GProj}^H_{\rm fpd}\mbox{-}R$.

 Similarly, we denote by ${\rm GProj}^{H_2}_{\rm fpd}\mbox{-}R_2$ the full subcategory of  ${\rm GProj}^{H_2}\mbox{-}R_2$ formed by those graded Gorenstein-projective $R_2$-modules whose underlying graded $S_2$-modules have finite projective dimension. We observe that the latter condition is equivalent to that the underlying graded $S$-modules have finite projective dimension; see \cite[Lemma~4.1]{Chen-Wu}. 

 \begin{thm}\label{thm:main}
 Keep the assumptions above. Then there is a Frobenius quotient
 $${\rm GProj}^{H_2}_{\rm fpd}\mbox{-}R_2 \longrightarrow  {\rm Inf}_{m, n}({\rm GProj}^H_{\rm fpd}\mbox{-}R),$$
 whose essential kernel equals ${\rm Add}\; \{R_2(h\vec{x}^i), R_2(h\vec{y}^j)\; |\; h\in H, 0\leq i\leq n, 0\leq j\leq m\}$.
 \end{thm}

 \begin{proof}
     Set $H_1$ to be the smallest subgroup of $H_2$ containing $H$ and $\vec{x}_1$. Set $R_1=S[x;\tau]/(x^{n+1}+\omega)$, which is naturally $H_1$-graded. By Proposition~\ref{prop:GProj}, there is a Frobenius quotient
     $$\Phi\colon  {\rm GProj}^{H_1}_{\rm fpd}\mbox{-}R_1\longrightarrow {\rm Inf}_n({\rm GProj}^{H}_{\rm fpd}\mbox{-}R),$$
     whose essential kernel equals ${\rm Add}\; \{R_1(h)\; |\; h\in H\}$. Applying Theorem~\ref{thm:inf}, we obtain the following Frobenius quotient
          $${\rm Inf}_m(\Phi) \colon  {\rm Inf}_m({\rm GProj}^{H_1}_{\rm fpd}\mbox{-}R_1) \longrightarrow {\rm Inf}_m({\rm Inf}_n({\rm GProj}^{H}_{\rm fpd}\mbox{-}R))={\rm Inf}_{m, n}({\rm GProj}^{H}_{\rm fpd}\mbox{-}R)),$$
          whose essential kernel equals ${\rm Add}\; \{\theta^j(R_1(h))\; |\; h\in H, 1\leq j\leq m\}$.

     We apply Proposition~\ref{prop:GProj} again by replacing $S$ by $S[x;\tau]$, $x$ by $y$, and $\omega$ by $x^{n+1}+\omega$. We obtain a Frobenius quotient
  $$\Phi_1\colon  {\rm GProj}^{H_2}_{\rm fpd}\mbox{-}R_2\longrightarrow {\rm Inf}_m({\rm GProj}^{H_1}_{\rm fpd}\mbox{-}R_1),$$
  whose essential kernel equals ${\rm Add}\; \{R_2(h_1)\; |\; h_1\in H_1\}$. By the explicit construction of $\Phi_1$, we observe that
  \begin{align}\label{equ:Phi_1}
      \Phi_1(R_2(h\vec{y}^j))\simeq \theta^{m+1-j}(R_1(h))
  \end{align}
for all $h\in H$ and $1\leq j\leq m$. By composing $\Phi_1$ and ${\rm Inf}_m(\Phi)$, we obtain the required Frobenius quotient. Moreover, by (\ref{equ:Phi_1}), we deduce the last statement about the essential kernel.
 \end{proof}

In view of Corollary~\ref{cor:fgl}, we have the following finite version of Theorem~\ref{thm:main}.

\begin{cor}\label{cor:main}
     Assume that the $H$-graded ring $S$ is right graded-noetherian of finite right graded global dimension. Then there is a Frobenius quotient
 $${\rm GProj}^{H_2}\mbox{-}R_2 \longrightarrow  {\rm Inf}_{m, n}({\rm GProj}^H\mbox{-}R),$$
 which restricts to a Frobenius quotient
 $${\rm Gproj}^{H_2}\mbox{-}R_2 \longrightarrow  {\rm Inf}_{m, n}({\rm Gproj}^H\mbox{-}R).$$
\end{cor}

 \begin{exm}\label{exm:V_4}
    {\rm We take $H={1_H}$ to be the trivial group and $S$ to be an ordinary ring.
Take $n=1=m$ and $\lambda_0$ to be $1$ or $-1$. The group $H_2$ is isomorphic to the Klein four group $V_4$. Then we have the following Frobenius quotient
$${\rm GProj}^{V_4}_{\rm fpd}\mbox{-} S_2/{(x^2+y^2+\omega)} \longrightarrow  {\rm GProj}_{\rm fpd}\mbox{-}S/{(\omega)},$$
which induces a triangle equivalence
$$\underline{\rm GProj}^{V_4}_{\rm fpd}\mbox{-} S_2/{(x^2+y^2+\omega)} \stackrel{\sim}\longrightarrow  \underline{\rm GProj}_{\rm fpd}\mbox{-}S/{(\omega)}.$$
This stable equivalence reminds us the famous Kn\"{o}rrer periodicity \cite{Kno}; compare \cite{CKMW, Chen-Wu}. However, they are very different, since the form of the Kn\"{o}rrer periodicity depends on the ground field.

Let $\mathbb{K}$ be a field. Take $S=\mathbb{K}[z]$ to be the polynomial algebra and $\omega=z^p$. Then $R=\mathbb{K}[z]/(z^p)$ is finite dimensional self-injective and every $\mathbb{K}[z]/(z^p)$-module is Gorenstein-projective. Set $\lambda_0=1$.  Then the Frobenius quotient in Corollary~\ref{cor:main} has the following form
 $${\rm MCM}^{V_4}({\mathbb{K}[x, y, z]/(x^2+y^2+z^p)}) \longrightarrow  {\rm mod}\mbox{-}\mathbb{K}[z]/(z^p).$$
 Here, we identify maximal Cohen-Macaulay $\mathbb{K}[x, y, z]/(x^2+y^2+z^p)$-modules with Gorenstein-projective $\mathbb{K}[x, y, z]/(x^2+y^2+z^p)$-modules. By the same argument, we have a Frobenius quotient
  \begin{align}\label{equ:V_4}
  {\rm MCM}^{V_4}({\mathbb{K}[[x, y, z]]/(x^2+y^2+z^p)}) \longrightarrow  {\rm mod}\mbox{-}\mathbb{K}[z]/(z^p).
  \end{align}
  We refer to Example~\ref{exm:final} for more details. 
}
 \end{exm}

\section{Maximal Cohen-Macaulay modules and weighted projective lines} \label{sec:WPL}

In this section, we apply the results to weighted projective lines with three weights. We fix a ground field $\mathbb{K}$ and three integers $p, q$ and $r$, which are at least $2$.

The truncated polynomial algebra $\mathbb{K}[t]/(t^r)$ is naturally $\mathbb{Z}$-graded by means of ${\rm deg}(t)=1$. Following Example~\ref{exm:grid}, we write $$\mathcal{S}^\mathbb{Z}_{p-1, q-1}(r)={\rm Inf}_{p-1, q-1}({\rm mod}^\mathbb{Z}\mbox{-}\mathbb{K}[t]/(t^r)),$$
which is the  category of $(p-1, q-1)$-inflations in ${\rm mod}^\mathbb{Z}\mbox{-}\mathbb{K}[t]/(t^r)$.

Denote by $L=L(p, q, r)$ the rank one abelian group generated by $\vec{x}, \vec{y}$ and $\vec{z}$, which are subject to the relations $p\vec{x}=q\vec{y}=r\vec{z}$.
This common element is denoted by $\vec{c}$, and is called the \emph{canonical element}. Denote by $H$ the cyclic subgroup of $L$ generated by $\vec{z}$. We have a disjoint union
$$L=\bigcup_{0\leq i\leq p-1, 0\leq j\leq q-1} H+i\vec{x}+j\vec{y}.$$

 Consider the following algebra
$$S(p, q, r)=\mathbb{K}[x, y, z]/(x^p+y^q+z^r),$$
which is naturally $L$-graded by means of ${\rm deg}(x)=\vec{x}$, ${\rm deg}(y)=\vec{y}$ and ${\rm deg}(z)=\vec{z}$.
We observe that the subalgebra
$$S(p, q, r)_H=\bigoplus_{\vec{l} \in H} S(p, q, r)_{\vec{l}}$$
equals $\mathbb{K}[x^p, z]$ and is isomorphic to a polynomial algebra in two variables.

Since $S(p, q, r)$ is graded Gorenstein, we have
$${\rm Gproj}^{L}\mbox{-}S(p, q,r)={\rm MCM}^L(S(p,q, r)).$$
Moreover, an $L$-graded $S(p,q, r)$-module $M$ is maximal Cohen-Macaulay if and only if for any $0\leq i\leq p-1, 0\leq j\leq q-1$,  the restriction
$$M_{H+i\vec{x}+j\vec{y}}=\bigoplus_{\vec{l}\in H+i\vec{x}+j\vec{y}} M_{\vec{l}}$$
is a graded projective module over $S(p, q, r)_H$; compare \cite[Lemma~3.3]{Chen12-Ld}. Moreover, we visualize such a graded module $M$ as follows.
 \[\xymatrix@R=15pt@C=15pt{
M_{H+(q-1)\vec{y}}  \ar[r]^-x & M_{H+\vec{x}+(q-1)\vec{y}}  \ar[r] & \cdots \ar[r] & M_{H+(p-2)\vec{x}+(q-1)\vec{y}} \ar[r]^x & M_{H+(p-1)\vec{x}+(q-1)\vec{y}} \\
M_{H+(q-2)\vec{y}} \ar[u]^y \ar[r]^-x & M_{H+\vec{x}+(q-2)\vec{y}}  \ar[u]^y \ar[r] & \cdots \ar[r] & M_{H+(p-2)\vec{x}+(q-2)\vec{y}} \ar[u]^y \ar[r]^x & M_{H+(p-1)\vec{x}+(q-2)\vec{y}} \ar[u]_-y \\
\vdots\ar[u]    & \vdots \ar[u]  & \cdots  & \vdots\ar[u]   & \vdots  \ar[u]\\
M_{H+\vec{y}} \ar[u] \ar[r]^-x & M_{H+\vec{x}+\vec{y}} \ar[u]  \ar[r] & \cdots \ar[r] & M_{H+(p-2)\vec{x}+\vec{y}} \ar[u]\ar[r]^-x & M_{H+(p-1)\vec{x}+\vec{y}}  \ar[u]\\
M_H \ar[u]^y\ar[r]^-x & M_{H+\vec{x}} \ar[u]^y \ar[r] & \cdots \ar[r] & M_{H+(p-2)\vec{x}}\ar[u]^-y \ar[r]^-x & M_{H+(p-1)\vec{x}} \ar[u]_-y
}\]
Here, the horizontal arrows denote the action of $x$ and the vertical ones denote the action of $y$. We omit the $x$-action on entries in the rightmost column and the $y$-action on entries in  the top row. For any $1\leq i\leq p-1$ and $1\leq j\leq q-1$, we denote by $C^{i,j}$ the cokernel of the following map
$$M_{H+i\vec{x}}\oplus M_{H+j\vec{y}}\xrightarrow{(y^j, \; x^i)}  M_{H+i\vec{x}+j\vec{y}}.$$
We observe that $C^{i,j}$ is a finite-dimensional $\mathbb{Z}$-graded $\mathbb{K}[t]/(t^r)$-module, where $t$ acts by $z$. Moreover, the actions of $x$ and $y$ on $M$ induce monomorphisms between these modules.
\[\xymatrix{
C^{i, j+1} \ar[rr]^-{\bar{x}} && C^{i+1, j+1}\\
C^{i, j}\ar[u]^-{\bar{y}} \ar[rr]^-{\bar{x}} && C^{i+1, j} \ar[u]_-{\bar{y}}
}\]
Indeed, these data form an object $(C^{i,j})$ in $\mathcal{S}^\mathbb{Z}_{p-1, q-1}(r)$. We will denote $(C^{i,j})$ by
${\rm DCok}(M)$, and call it the \emph{double-cokernel} of $M$. This certainly gives rise to a well-defined functor
$${\rm DCok}\colon  {\rm MCM}^L(S(p,q, r)) \longrightarrow \mathcal{S}^\mathbb{Z}_{p-1, q-1}(r).$$

The following result is essentially a special case of Theorem~\ref{thm:main}.

\begin{prop}\label{prop:WPL}
    Keep the notation above. Then the functor ${\rm DCok}$ is a Frobenius quotient, whose essential kernel equals
    $${\rm add}\; \{S(p, q, r)(\vec{l})\; |\; \vec{l}\in \bigcup_{0\leq i\leq p-1, 0\leq j\leq q-1} (H+i\vec{x} \cup H+j\vec{y})\}.$$
 In particular, we have an induced stable equivalence
 $$ \underline{\rm MCM}^L(S(p,q, r)) \stackrel{\sim}\longrightarrow \underline{\mathcal{S}}^\mathbb{Z}_{p-1, q-1}(r). $$
\end{prop}

\begin{proof}
    Set $H_1=H+\mathbb{Z}\vec{x}$. We closely follow the proof of Theorem~\ref{thm:main}. Let $M$ be any graded maximal Cohen-Macaulay module over $S(p, q,r)$. Then $M$ is identified with the following  $q$-factorization of $x^p+z^r$ over $\mathbb{K}[x, z]$.
    $$M_{H_1}\stackrel{y}\longrightarrow M_{H_1+\vec{y}} \longrightarrow \cdots \longrightarrow  M_{H_1+(q-2)\vec{y}}\stackrel{y}\longrightarrow M_{H_1+(q-1)\vec{y}} \stackrel{y}\longrightarrow  M_{H_1}(\vec{c})$$
    Applying $\Phi_1$ therein, we obtain a $(q-1)$-inflation
    $$\overline{M}_{H_1+\vec{y}} \stackrel{\bar{y}}\longrightarrow \overline{M}_{H_1+2\vec{y}}  \longrightarrow \cdots \longrightarrow \overline{M}_{H_1+(q-2)\vec{y}}  \stackrel{\bar{y}}\longrightarrow \overline{M}_{H_1+(q-1)\vec{y}}. $$
Here, each component $\overline{M}_{H_1+j\vec{y}} $  is the cokernel of the map
$$y^j\colon M_{H_1}\longrightarrow M_{H_1+j\vec{y}}.$$

   For each  $1\leq j\leq q-1$,  the component $\overline{M}_{H_1+j\vec{y}} $ is viewed as the following $(p-1)$-factorization of $z^r$  in $\mathbb{K}[z]$.
        $$ \overline{M}_{H+j\vec{y}}  \stackrel{x}\longrightarrow \overline{M}_{H+\vec{x}+j\vec{y}} \longrightarrow \cdots \longrightarrow  \overline{M}_{H+(p-2)\vec{x}+j\vec{y}}\stackrel{x}\longrightarrow\overline{M}_{H+(p-1)\vec{x}+j\vec{y}}\stackrel{x}\longrightarrow  M_{H+j\vec{y}}(\vec{c})$$
        Applying $\Phi$ therein to this component, we obtain a $(p-1)$-inflation in ${\rm mod}^\mathbb{Z}\mbox{-}\mathbb{K}[t]/(t^r)$:
        $$C^{1, j}\stackrel{\bar{x}}\longrightarrow C^{2, j} \longrightarrow \cdots \longrightarrow C^{p-2, j}\stackrel{\bar{x}}\longrightarrow C^{p-1, j}.$$
        Here, we implicitly use the fact that $C^{i, j}$ is isomorphic to the cokernel of the following map
        $$x^i\colon \overline{M}_{H+j\vec{y}}\longrightarrow  \overline{M}_{H+i\vec{x}+j\vec{y}}.$$

        Letting $j$ vary, we obtain the $(p-1, q-1)$-inflation $(C^{i,j})$ above. In other words, the composition ${\rm Inf}_{q-1}(\Phi)\circ \Phi_1$ coincides with ${\rm DCok}$. Then the required statement  follows immediately from Theorem~\ref{thm:main} and Corollary~\ref{cor:main}.
\end{proof}

\begin{rem}
    We observe the symmetry among the parameters $p, q$ and $r$. Then the stable equivalence above implies the following stable equivalences,
    $$\underline{\mathcal{S}}^\mathbb{Z}_{p-1, q-1}(r)\simeq  \underline{\mathcal{S}}^\mathbb{Z}_{p-1, r-1}(q)\simeq \underline{\mathcal{S}}^\mathbb{Z}_{q-1, r-1}(p)$$
    which might be viewed as the  \emph{Happel-Seidel-type symmetry}; compare \cite[Theorem~6.11]{KLM1} and Example~\ref{exm:pq}.
\end{rem}

Denote by $\mathbb{X}(p, q,r)$ the \emph{weighted projective line} \cite{GL} with weight sequence $(p,q,r)$. The structure sheaf is denoted by $\mathcal{O}$. The category of vector bundles over $\mathbb{X}(p, q,r)$ is denoted by ${\rm vect}\mbox{-}\mathbb{X}(p, q,r)$.

By \cite[Theorem~5.8]{GL}, the sheafification yields an equivalence
$$\widetilde{(-)}\colon {\rm MCM}^L(S(p,q, r)) \longrightarrow  {\rm vect}\mbox{-}\mathbb{X}(p, q,r),$$
which sends $S(p, q, r)(\vec{l})$ to $\mathcal{O}(\vec{l})$ for each $\vec{l}\in L$. Moreover, it transfers the exact structure on  ${\rm MCM}^L(S(p,q, r))$ to ${\rm vect}\mbox{-}\mathbb{X}(p, q,r)$. To be more precise, a conflation in ${\rm vect}\mbox{-}\mathbb{X}(p, q,r)$ corresponds to the \emph{distinguished short exact sequences} of vector bundles \cite{KLM1}, that is, those short exact sequences
$$\eta\colon 0\longrightarrow \mathcal{F}_1\longrightarrow \mathcal{F}_2\longrightarrow \mathcal{F}_3\longrightarrow 0$$
of vector bundles such that ${\rm Hom}(\mathcal{O}(\vec{l}), \eta)$ is exact for any $\vec{l}\in L$. Equipped with such conflations, ${\rm vect}\mbox{-}\mathbb{X}(p, q,r) $ becomes a Frobenius category, whose projective-injectives are precisely finite direct sums of line bundles.

In view of Example~\ref{exm:grid}, the following result is essentially announced in \cite[p.198]{KLM1} and mentioned in  \cite[Theorem~4.4]{Sim15}. The setting there is  quite different, and there is  no detailed proof in the literature. 

\begin{thm}\label{thm:WPL}
    The composition of ${\rm DCok}$ and a quasi-inverse of $\widetilde{(-)}$ yields a Frobenius quotient
    $$ {\rm vect}\mbox{-}\mathbb{X}(p, q,r) \longrightarrow {\mathcal{S}}^\mathbb{Z}_{p-1, q-1}(r), $$
    whose essential kernel equals
     $${\rm add}\; \{\mathcal{O}(i\vec{x}+n\vec{z}), \mathcal{O}(j\vec{y}+n\vec{z})\; |\; 0\leq i\leq p-1, 0\leq j\leq q-1, n\in \mathbb{Z}\}.$$
\end{thm}

\begin{proof}
Set $S=S(p, q ,r)$.    The sheafification $\widetilde{(-)}$ above is an exact equivalence, which sends $S(\vec{l})$ to $\mathcal{O}(\vec{l})$. Now the required statements follow immediately from Proposition~\ref{prop:WPL}. 
\end{proof}

\begin{rem}
    If $(p,q,r)=(2,3,r)$, the Frobenius quotient above is of the form
    $${\rm vect}\mbox{-}\mathbb{X}(2, 3,r) \longrightarrow {\mathcal{S}}^\mathbb{Z}_{1, 2}(r)={\mathcal{S}^\mathbb{Z}}(r),$$
    with essential kernel 
    $$\textup{add}\; \{\mathcal{O}(i\vec{x}+j\vec{y}+n\vec{z})\; |\; (i,j)=(0,0),(0,1),(0,2),(1,0),\;\textup{and}\;n\in\mathbb{Z}\}.$$
    Up to a degree-shift by $\vec{x}+2\vec{y}$, the line bundles in the essential kernel coincide  with the \emph{fading line bundles} in \cite{KLM2}. Thus, this Frobenius quotient recovers the main result of \cite{KLM2}. The original functor presented in \cite{KLM2} is quite mysterious; compare \cite{Chen12}. In contrast, the functor in Theorem~\ref{thm:WPL} is more explicit.
\end{rem}

\begin{exm}\label{exm:final}
{\rm Set $S=S(2,2,p)$ and $L=L(2,2,p)$. We identify $L/{\mathbb{Z}\vec{z}}$ with the Klein four group $V_4$. Set $\widehat{S}=\mathbb{K}[[x, y, z]]/(x^2+y^2+z^p)$. Consider the following commutative diagram. 
\[\xymatrix{ 
{\rm MCM}^L(S) \ar[d]_-{c}\ar[rr]^-{\rm DCok} && {\rm mod}^\mathbb{Z}\mbox{-}{\mathbb{K}[z]/(z^p)}\ar[d]^-{U}\\
{\rm MCM}^{V_4}(\widehat{S}) \ar[rr]^-{\rm DCok} && {\rm mod}\mbox{-}{\mathbb{K}[z]/(z^p)}
}\]
The lower Frobenius quotient is studied in Example~\ref{exm:V_4}, which is also given by a similar double-cokernel construction. We denote by $U$ the forgetful functor, and by $c$ the \emph{completion} functor \cite{AR}.

By using  smash products in \cite{CM} and \cite[Theorem~3.8]{HY}, the completion functor $c$ preserves indecomposability and Auslander-Reiten sequences. Moreover, by \cite[Theorem~3.2(a)]{AR}, it is dense.  Consequently, it induces an isomorphism  of quivers
$$\Gamma(\textup{MCM}^{L}(S))/[\mathbb{Z}\vec{z}] \stackrel{\sim}\longrightarrow \Gamma (\textup{MCM}^{V_4} (\widehat{S})).$$
Here, both $\Gamma(\textup{MCM}^{L}(S))$ and $\Gamma (\textup{MCM}^{V_4} (\widehat{S}))$ denote the corresponding Auslander-Reiten quivers.  By $\Gamma(\textup{MCM}^{L}(S))/[\mathbb{Z}\vec{z}]$, we mean the quotient quiver modulo the $\vec{z}$-action.

We identify $ {\rm MCM}^L(S)$ with ${\rm vect}\mbox{-}\mathbb{X}(2,2, p)$.  It is well known that its Auslander-Reiten quiver has the shape $\mathbb{Z}\widetilde{D}_{p+2}$. Under the action of $\vec{z}$ on $\textup{vect}\mbox{-}\mathbb{X}(2,2, p)$, we infer that the quotient quiver $\Gamma(\textup{MCM}^{L}(S))/[\mathbb{Z}\vec{z}]$ is given by the  double quiver of $\widetilde{D}_{p+2}$, and the projective-injective objects correspond to the four endpoints. Consequently, the Auslander-Reiten quiver of the stable category $\underline{\textup{MCM}}^{V_4}(\widehat{S})$ is given by the  double quiver of $A_{p-1}$.

In what follows, we take $p=3$ to illustrate the argument. The Auslander-Reiten quiver of $\textup{vect}\mbox{-}\mathbb{X}(2,2,3)$ is as follows. 
\begin{figure}[h]
    \begin{center}
\begin{tikzpicture}
\node(q5)at(-0.8,-1){$\cdots$};
    \node(q6)at(-0.8,-0.4){$\cdots$};
    \node(q3)at(-0.8,1){$\cdots$};
    \node(02)at(0,2){\tiny{$\diamondsuit$}};
    \node(01)at(0,1.4){\tiny{$\square$}};
    \node(04)at(0,0){\tiny{$\triangle$}};
    \node(13)at(0.8,1){\tiny{$\bigtriangledown$}};
    \node(15)at(0.8,-0.4){$\circ$};
    \node(16)at(0.8,-1){$\bullet$};
    \node(24)at(1.6,0){\tiny{$\triangle$}};
    \node(21)at(1.6,1.4){\tiny{$\diamondsuit$}};
    \node(22)at(1.6,2){\tiny{$\square$}};
    \node(33)at(2.4,1){\tiny{$\bigtriangledown$}};
    \node(35)at(2.4,-0.4){$\bullet$};
    \node(36)at(2.4,-1){$\circ$};
    \node(44)at(3.2,0){\tiny{$\triangle$}};
    \node(41)at(3.2,1.4){\tiny{$\square$}};
    \node(42)at(3.2,2){\tiny{$\diamondsuit$}};
    \node(53)at(4,1){\tiny{$\bigtriangledown$}};
    \node(55)at(4,-0.4){$\circ$};
    \node(56)at(4,-1){$\bullet$};
     \node(64)at(4.8,0){$\cdots$};
    \node(61)at(4.8,1.4){$\cdots$};
    \node(62)at(4.8,2){$\cdots$};
    \draw[->](q3)to(01);
    \draw[->](q3)to(02);
    \draw[->](04)to(13);
    \draw[->](04)to(15);
    \draw[->](04)to(16);
    \draw[<-](04)to(q3);
    \draw[<-](04)to(q5);
    \draw[<-](04)to(q6);
    \draw[<-](24)to(13);
    \draw[<-](24)to(15);
    \draw[<-](24)to(16);
    \draw[->](13)to(21);
    \draw[->](13)to(22);
    \draw[<-](13)to(01);
    \draw[<-](13)to(02);
    \draw[->](24)to(35);
    \draw[->](24)to(36);
    \draw[->](24)to(33);
    \draw[->](21)to(33);
    \draw[->](22)to(33);
    \draw[<-](44)to(35);
    \draw[<-](44)to(36);
    \draw[<-](44)to(33);
    \draw[<-](41)to(33);
    \draw[<-](42)to(33);
     \draw[->](44)to(55);
    \draw[->](44)to(56);
    \draw[->](44)to(53);
    \draw[->](41)to(53);
    \draw[->](42)to(53);
    \draw[<-](64)to(55);
    \draw[<-](64)to(56);
    \draw[<-](64)to(53);
    \draw[<-](61)to(53);
    \draw[<-](62)to(53);
\end{tikzpicture}
    \end{center}
    \caption{The AR quiver of $\textup{vect}\mbox{-}\mathbb{X}(2,2,3)$}\label{AR1}
\end{figure}
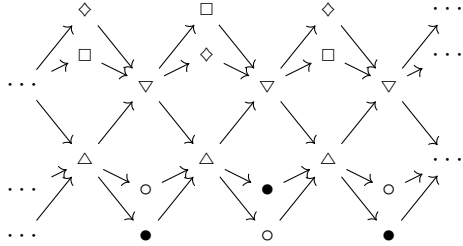

The $\mathbb{Z}\vec{z}$-orbits of $\mathcal{O},\mathcal{O}(\vec{x}),\mathcal{O}(\vec{y}),\mathcal{O}(\vec{x}+\vec{y})$ are marked by the symbols $\bullet,\tiny{\diamondsuit},\tiny{\square},\circ$, respectively, and the $\mathbb{Z}\vec{z}$-orbits of 2 rank-two indecomposable vector bundles are marked by $\tiny{\triangle}$ and $\tiny{\bigtriangledown}$. They form the  projective-injective objects. The quotient quiver $\Gamma(\textup{MCM}^{L}(S))/[\mathbb{Z}\vec{z}]$ is  Figure~\ref{AR2}, which is also the Auslander-Reiten quiver of $\textup{MCM}^{V_4}(\widehat{S})$.
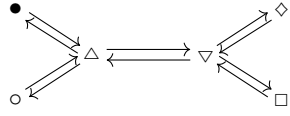
\begin{figure}[h]
    \begin{center}
\begin{tikzpicture}
    \node(02)at(2.5,0.6){\tiny{$\diamondsuit$}};
    \node(01)at(2.5,-0.6){\tiny{$\square$}};
    \node(04)at(0,0){\tiny{$\triangle$}};
    \node(13)at(1.5,0){\tiny{$\bigtriangledown$}};
    \node(15)at(-1,-0.6){$\circ$};
    \node(16)at(-1,0.6){$\bullet$};
    \draw[->](0.2,0.07)--(1.3,0.07);
    \draw[<-](0.2,-0.07)--(1.3,-0.07);
    \draw[->] (-0.87,-0.46) -- (-0.18,-0.02);
    \draw[->] (-0.15,-0.12) -- (-0.82,-0.56);
    \draw[<-] (-0.87,0.46) -- (-0.18,0.02);
    \draw[<-] (-0.15,0.12) -- (-0.82,0.56);
    \draw[->] (2.37,0.46) -- (1.68,0.02);
    \draw[->] (1.65,0.12) -- (2.32,0.56);  
    \draw[<-] (2.37,-0.46) -- (1.68,-0.02);
    \draw[<-] (1.65,-0.12) -- (2.32,-0.56);  
\end{tikzpicture}
    \end{center}
    \caption{The AR quiver of $\textup{MCM}^{V_4}(\mathbb{K}[[x, y,z]]/(x^2+y^2+z^3))$}\label{AR2}
\end{figure}

Recall the Frobenius quotient $\textup{MCM}^{V_4}(\widehat{S})\rightarrow {\rm mod}\mbox{-}\mathbb{K}[z]/(z^3)$. Therefore, the Auslander-Reiten quiver of ${\rm mod}\mbox{-}\mathbb{K}[z]/(z^3)$ is obtained from the one of $\textup{MCM}^{V_4}(\widehat{S})$ by deleting the projective-injective objects in the essential kernel. These objects are precisely the ones marked by the symbols $\bullet,\tiny{\diamondsuit},\tiny{\square}$. 

\begin{figure}[h]
    \begin{center}
        \begin{tikzpicture}
          \node()at(0,0){\tiny{$\triangle$}};
          \node()at(1.2,0){\tiny{$\bigtriangledown$}};
          \node()at(-1.2,0){{$\circ$}};
          \draw[->] (-1,0.07)--(-0.2,0.07);
          \draw[->] (-1,0.07)--(-0.2,0.07);
          \draw[<-] (-1,-0.07)--(-0.2,-0.07);
          \draw[<-] (-1,-0.07)--(-0.2,-0.07);
          \draw[<-] (1,0.07)--(0.2,0.07);
          \draw[<-] (1,0.07)--(0.2,0.07);
          \draw[->] (1,-0.07)--(0.2,-0.07);
          \draw[->] (1,-0.07)--(0.2,-0.07);
      \end{tikzpicture}
    \end{center}
    \caption{The AR quiver of $\textup{mod}\mbox{-}\mathbb{K}[z]/(z^3)$}
    \label{AR3}
\end{figure}
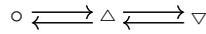
Consequently, the Auslander-Reiten quiver of ${\rm mod}\mbox{-}\mathbb{K}[z]/(z^3)$ is given by the double quiver of $A_{3}$, as shown in Figure \ref{AR3}. Here, the projective-injective module corresponds to the symbol $\circ$.

}
\end{exm}

\vskip 5pt

\noindent {\bf Acknowledgements.} The authors are very grateful to Dr. Xiaofa Chen for pointing out the references \cite{Bar,War}. 
This project is supported by National Key R$\&$D Program of China (No. 2024YFA1013801), the National Natural Science Foundation of China (No.s 12325101, 12131015, 12301054,  12271448 and 12471035)  and the Fujian Provincial Natural Science Foundation of China (No. 2024J010006).

\bibliography{}

\vskip 10pt

 {\footnotesize \noindent  Xiao-Wu Chen\\
 School of Mathematical Sciences, University of Science and Technology of China, Hefei 230026, Anhui, PR China}

\vskip 3pt

{\footnotesize \noindent Qiang Dong\\
School of Mathematics and Statistics, Fujian Normal University, Fuzhou 350117, Fujian, PR China}

\vskip 3pt

{\footnotesize \noindent Shiquan Ruan\\
School of Mathematical Sciences, Xiamen University, Xiamen 361005, PR China}

\end{document}